\documentclass[11pt, leqno]{amsart}

\usepackage{texdraw,multicol,rotating}

 \usepackage{youngtab}

\usepackage{amssymb}
\usepackage{enumerate}
\usepackage[all]{xy}
\usepackage{verbatim}

\let\goth\mathfrak

\def\gl{\goth l}

\def\beq{\begin{equation*}}
\def\eeq{\end{equation*}}
\def\bea{\begin{aligned}}
\def\eea{\end{aligned}}
\def\bee{\begin{enumerate}}
\def\eee{\end{enumerate}}

\def\cplus{\hbox{$\supset${\raise1.05pt\hbox{\kern -0.55em
${\scriptscriptstyle +}$}}\ }}

\numberwithin{equation}{section}

\theoremstyle{plain}
\newtheorem{theorem}{Theorem}[section]
\newtheorem{proposition}[theorem]{Proposition}
\newtheorem{lemma}[theorem]{Lemma}
\newtheorem{corollary}[theorem]{Corollary}

\theoremstyle{definition}
\newtheorem{definition}[theorem]{Definition}

\newtheorem{example}[theorem]{Example}
\newtheorem*{remark}{Remark}

\theoremstyle{remark}

\newtheorem*{acknowledgement}{Acknowledgements}

\def\N{\mathbb N}

\def\P{\mathbb P}

\def\B{\mathbf{B}}
\def\P{\mathbf{P}}

\def\gl{\goth{gl}}

\def\ot{\otimes}

\def\ra{\longrightarrow}
\def\map{\longmapsto}

\allowdisplaybreaks

\def\sh{{\rm sh}}

\newcommand{\sqo}{\bsegment \move(0
0)\lvec(10 0)\lvec(10 10)\lvec(0 10)\lvec(0 0) \htext(3 3){1}
\esegment \rmove(10 0)}

\newcommand{\dsq}{\bsegment \lpatt(0.5 1) \move(0
0)\lvec(10 0)\lvec(10 10)\lvec(0 10)\lvec(0 0) \esegment \rmove(10
0)}

\newcommand{\sq}{\bsegment \move(0
0)\lvec(10 0)\lvec(10 10)\lvec(0 10)\lvec(0 0) \esegment \rmove(10
0)}

\newcommand{\bsq}{\bsegment \move(0
0)\lvec(10 0)\lvec(10 10)\lvec(0 10)\lvec(0 0) \lfill f:0.7
\esegment \rmove(10 0)}

\newcommand{\sqt}{ \bsegment \move(0
0)\lvec(10 0)\lvec(10 10)\lvec(0 10)\lvec(0 0) \htext(3 3){2}
\esegment \rmove(10 0)}

\newcommand{\sqth}{ \bsegment \move(0
0)\lvec(10 0)\lvec(10 10)\lvec(0 10)\lvec(0 0) \htext(3 3){3}
\esegment \rmove(10 0)}

\newcommand{\sqf}{ \bsegment \move(0
0)\lvec(10 0)\lvec(10 10)\lvec(0 10)\lvec(0 0) \htext(3 3){4}
\esegment \rmove(10 0)}

\newcommand{\sqfi}{ \bsegment \move(0
0)\lvec(10 0)\lvec(10 10)\lvec(0 10)\lvec(0 0) \htext(3 3){5}
\esegment \rmove(10 0)}

\newcommand{\sqbo}{ \bsegment \move(0
0)\lvec(10 0)\lvec(10 10)\lvec(0 10)\lvec(0 0) \htext(1
2){\tiny{$b_1$}} \esegment \rmove(10 0)}

\newcommand{\sqbt}{ \bsegment \move(0
0)\lvec(10 0)\lvec(10 10)\lvec(0 10)\lvec(0 0) \htext(1
2){\tiny{$b_2$}} \esegment \rmove(10 0)}

\newcommand{\sqbn}{ \bsegment \move(0
0)\lvec(10 0)\lvec(10 10)\lvec(0 10)\lvec(0 0) \htext(1
2){\tiny{$b_N$}} \esegment \rmove(10 0)}

\begin{document}
\title[Admissible Pictures and $U_q(\gl(m,n))$-Littlewood-Richardson Tableaux]
{Admissible Pictures and $U_q(\gl(m,n))$-\\  Littlewood-Richardson
Tableaux}
\author[Jung, Kang, Lyoo]{Ji Hye Jung$^{1,2}$, Seok-Jin Kang$^{1,3}$ and Young-Wook Lyoo }

\address{Department of Mathematical Sciences
         and
         Research Institute of Mathematics \\
         Seoul National University \\ San 56-1 Sillim-dong, Gwanak-gu \\ Seoul 151-747, Korea}

         \email{jhjung@math.snu.ac.kr}

\address{Department of Mathematical Sciences
         and
         Research Institute of Mathematics \\
         Seoul National University \\ San 56-1 Sillim-dong, Gwanak-gu \\ Seoul 151-747, Korea}

         \email{sjkang@math.snu.ac.kr}

\address{Seoul Science High School \\ Wooam-Gil 63,  Jongro-gu \\ Seoul 110-530, Korea }

         \email{yw.lyoo@gmail.com}

\thanks{$^{1}$This research was supported by KRF Grant \# 2007-341-C00001.}
\thanks{$^{2}$This research was supported by BK21 Mathematical Sciences Division.}
\thanks{$^3$This research was supported by National Institute for Mathematical Sciences (2010 Thematic Program, TP1004).}

\begin{abstract}

We construct a natural bijection between the set of admissible
pictures and the set of $U_q(\gl(m,n))$-Littlewood-Richardson
tableaux.

\end{abstract}

\maketitle

\section*{Introduction}

The notion of {\it pictures} was introduced by James and Peel
\cite{JP} and Zelevinsky \cite{Z}. In \cite{NS}, Nakashima and
Shimojo considered the notion of {\it admissible pictures} and
showed that there exists a natural bijection between the set of
admissible pictures and the set of
$U_q(\gl(r))$-Littlewood-Richardson tableaux. More precisely, let
$Y$, $W$, $Z$ be Young diagrams with at most $r$ rows such that
$|Y|+|W|=|Z|$, and let $A$, $A'$ be admissible orders on $Z/Y$,
$W$, respectively. Then Nakashima and Shimojo constructed an
explicit natural bijection between the set $\P(W, Z/Y; A, A')$ of
$(A, A')$-admissible pictures and the set $\B(W)_{Y}^{Z}[A']$ of
$U_q(\gl(r))$-Littlewood-Richardson tableaux. This result was
already obtained in \cite{CS} and \cite{FG} by a purely
combinatorial method. Nakashima and Shimojo gave an alternative
proof using the theory of $U_q(\gl(r))$-crystals.

In this paper, we generalize the main result of \cite{NS} to the
case when $W$ is a skew Young diagram (Theorem
\ref{thm:skew_case}). Our proof follows the outline given in
\cite{NS} with some necessary modifications to deal with
semistandard skew tableaux.

Moreover, we introduce the notion of {\it
$U_q(\gl(m,n))$-Littlewood-Richardson tableaux} arising from the
theory of $U_q(gl(m,n))$-crystals, and show that there exists a
natural bijection between the set of admissible pictures and the
set of $U_q(\gl(m,n))$-Littlewood-Richardson tableaux (Theorem
\ref{thm:main}). Namely, let $Y$, $W$, $Z$ be $(m,n)$-hook Young
diagrams such that $|Y|+|W|=|Z|$ and let $A$, $A'$ be admissible
orders on $W$, $Z/Y$, respectively. Denote by $LR(Y,W)^{Z}[A']$
the set of $U_q(\gl(m,n))$-Littlewood-Richardson tableaux
associated with $(Y,W,Z)$ and $A'$. We construct an explicit
natural bijection between $\P(Z/Y,W; A, A')$ and $LR(Y,
W)^{Z}[A']$. As a corollary, we show that the
$U_q(\gl(m,n))$-Littlewood-Richardson coefficients and the
$U_q(\gl(r))$-Littlewood-Richardson coefficients are the same.

\vskip 1cm


\section{Admissible Pictures}

We first recall some basic notions that are used in this paper. A
{\it Young diagram} is a collection of boxes arranged in
left-justified rows with a weakly decreasing number of boxes in
each row as we go down. A Young diagram may be identified with a
{\it partition} $Y=(Y_1 \ge Y_2 \ge \cdots )$, where $Y_i$ is the
number of boxes in the $i$-th row. For Young diagrams $Y$ and $Z$
with $Z \supset Y$, we denote by $Z / Y$ the {\it skew Young
diagram} obtained by removing $Y$ from $Z$. For a skew Young
diagram $Y$, the {\it size} of $Y$, denoted by $|Y|$, is defined
to be the total number of boxes in $Y$. A {\it Young tableau}
(resp. {\it skew tableau}) $T$ is a filling of a Young diagram
(resp. skew Young diagram) with positive integers. We say that $T$
is {\it semistandard} if

\ \ (i) the entries in each row are weakly increasing from left to
right,

\ \ (ii) the entries in each column are strictly increasing from
top to bottom.

\noindent The (skew) Young diagram $Y$ is called the {\it shape}
of $T$. We often write $\sh(T)=Y$.

The notion of {\it pictures} was first introduced by James and
Peel \cite{JP} and Zelevinsky \cite{Z}. In this paper, we use a
generalized definition given by Nakashima and Shimojo \cite{NS}.

\begin{definition}
 Let $X$ and $Y$ be subsets of $\N \times \N$, where $\N$ is the set
 of positive integers.

(a) For $(a,b), (c,d) \in X$, we define
$$(a,b) \le_{P} (c,d) \ \ \text{if and only if} \ \ a \le c, \ b \le
d.$$

(b) A total order $\le_{A}$ on $X$ is said to be {\it admissible}
if
$$(a, b) \le_{A} (c,d) \ \ \text{whenever} \ \ a \le c, b \ge d.$$

\end{definition}

\noindent Thus $(a, b) \le_{A} (c, d)$ whenever $(a, b)$ lies in
the northeast of $(c, d)$.

\begin{example}
In this example, we introduce two typical examples of admissible
order.

(a) The {\it Middle Eastern order} $\le_{ME}$ on $X$ is defined by
$$(a, b) \le_{ME} (c,d)  \ \ \text{if and only if} \ \ a < c, \
\text{or} \ a=c, b \ge d.$$

(b) The {\it Far Eastern order} $\le_{FE}$ on $X$ is defined by
$$(a, b) \le_{FE} (c,d)  \ \ \text{if and only if} \ \ b > d, \
\text{or} \ b=d, a \le c.$$

\end{example}

Note that a (skew) Young diagram $Y$ may be regarded as a subset
of $\N \times \N$ by identifying the box in the $i$-th row and
$j$-th column with $(i,j) \in \N \times \N$. Hence we may consider
the notion of admissible orders on $Y$.

\begin{definition} Let $X$, $Y$ be subsets of $\N \times \N$ and
let $\le_{A}$ (resp. $\le_{A'}$) be an admissible order on $Y$
(resp. on $X$).

(a) A map $f:X \to Y$ is {\it $PA$-standard} if $f(a,b) \le_{A}
f(c,d)$ whenever $(a, b) \le_{P} (c,d)$.

(b) A bijection $f:X \to Y$ is called an {\it $(A, A')$-admissible
picture} if $f:X \to Y$ is $PA$-standard and $f^{-1}: Y \to X$ is
$PA'$-standard.

\end{definition}

We denote by $\P(X, Y; A, A')$ the set of all $(A, A')$-admissible
pictures from $X$ to $Y$. Since (skew) Young diagrams may be
considered as subsets of $\N \times \N$, for any pair of (skew)
Young diagrams $Y$ and $W$ with $|Y|=|W|$, we may define the
notion of admissible pictures from $Y$ to $W$ and vice versa.

\begin{example}
Let $X=\{(1,1), (1,2), (2,1), (2,2)\}$ and $Y=\{(1,3), (1,4),
(2,2), (2,3) \}$ be skew Young diagrams. Pictorially, we have $X=$
\raisebox{-0.5 \height}{
\begin{texdraw}
\drawdim em \setunitscale 0.15  \linewd 0.4    \sq \sq \move (0
10) \sq \sq
\end{texdraw}}  , $Y=$ \raisebox{-0.5\height}{
\begin{texdraw}
\drawdim em \setunitscale 0.15  \linewd 0.4 \dsq \dsq \sq \sq
\move (0 -10) \dsq \sq \sq
\end{texdraw}} .

Define a map $f:X \to Y$ by
$$f(1,1)=(1,4), \  f(1,2) = (1,3), \ f(2,1) = (2,3), \
f(2,2)=(2,2).$$ Then it is easy to verify that $f$ is an $(A,
A')$-admissible picture for any admissible orders $A$ and $A'$.
\end{example}

Let $f:X \to Y$ be an $(A, A')$-admissible picture. We denote by
$f_1$ and $f_2$ the 1st and 2nd coordinate functions,
respectively. That is, if $f(i,j) = (a,b)$, then $f_1(i,j) = a$,
$f_2(i,j)=b$. For simplicity, we often write $f = (f_1, f_2)$.

\vskip 1cm


\section{$U_q(\gl(r))$-Littlewood-Richardson tableaux}

In this section, we review the main result of \cite{NS} on the 1-1
correspondence between the set of admissible pictures and the set
of $U_q(\gl(r))$-Littlewood-Richardson tableaux. More details on
$U_q(\gl(r))$-crystals can be found in \cite{HK}.

Let $$\B: \ \ \boxed{1} \overset{1} \longrightarrow \boxed{2}
\overset{2} \longrightarrow \cdots \overset{n-2} \longrightarrow
\boxed{n-1} \overset{n-1} \longrightarrow \boxed{n} $$ be the
crystal of the vector representation of $U_q(\gl(r))$. For a
(skew) Young diagram $Y$, we denote by $\B(Y)$ the set of all
semistandard tableaux of shape $Y$ with entries in $\{1, 2,
\cdots, r \}$. If $|Y|=N$, a semistandard tableau $T \in \B(Y)$
can be identified with the element
$$\text{\raisebox{-0.2\height}{
\begin{texdraw}
\drawdim em \setunitscale 0.1  \linewd 0.4 \sqbo \end{texdraw}}
$\ot$  \raisebox{-0.2\height}{
\begin{texdraw}
\drawdim em \setunitscale 0.1  \linewd 0.4 \sqbt \end{texdraw}}
$\ot$ $\cdots$ $\ot$ \raisebox{-0.2\height}{
\begin{texdraw}
\drawdim em \setunitscale 0.1  \linewd 0.4 \sqbn \end{texdraw}}
$\in \B^{\ot N}$},$$
 where $b_1 , b_2 , \cdots , b_{N}$ are the entries of $T$ listed by a given admissible order
$A$. Thus we get an embedding
$$R_{A}: \B(Y) \rightarrow \B^{\ot N} \quad \text{given by} \quad T \mapsto
\text{\raisebox{-0.2\height}{
\begin{texdraw}
\drawdim em \setunitscale 0.1  \linewd 0.4 \sqbo \end{texdraw}}
$\ot$  \raisebox{-0.2\height}{
\begin{texdraw}
\drawdim em \setunitscale 0.1  \linewd 0.4 \sqbt \end{texdraw}}
$\ot$ $\cdots$ $\ot$ \raisebox{-0.2\height}{
\begin{texdraw}
\drawdim em \setunitscale 0.1  \linewd 0.4 \sqbn \end{texdraw}}
},$$ which is called the {\it admissible reading of $\B(Y)$ by
$A$}.

For example, the {\it Middle-Eastern reading} $R_{ME}(T)$ of $T
\in \B(Y)$ reads the entries by moving across the  rows from right
to left and top to bottom. On the other hand, the {\it Far-Eastern
reading} $R_{FE}(T)$ of $T \in \B(Y)$ proceeds down the columns
from top to bottom and from right to left.

\begin{example}
  \beq \bea
 R_{ME}\left( \raisebox{-0.3\height}{
\begin{texdraw}
\drawdim em \setunitscale 0.15  \linewd 0.4  \sqt \sqt \sqfi
\move(0 -10) \sqth \sqf \end{texdraw}} \, \right) =
  \boxed{5} \ot \boxed{2} \ot \boxed{2} \ot \boxed{4} \ot \boxed{3}
\eea \eeq

  \beq \bea
R_{FE}\left( \raisebox{-0.3\height}{
\begin{texdraw}
\drawdim em \setunitscale 0.15  \linewd 0.4  \sqt \sqt \sqfi
\move(0 -10) \sqth \sqf \end{texdraw}} \, \right) =  \boxed{5} \ot
\boxed{2} \ot \boxed{4} \ot \boxed{2} \ot \boxed{3}. \eea \eeq

\end{example}

An admissible reading $R_{A}$ provides $\B(Y)$ with a
$U_q(\gl(r))$-crystal structure by the tensor product rule. In
\cite{HK}, it was shown that the $U_q(\gl(r))$-crystal structure
on $\B(Y)$ does not depend on the choice of admissible reading
$R_{A}$.

Let $Y$ be a Young diagram. We denote by $Y[j]$ the diagram
obtained from $Y$ by adding a box at the $j$-th row. If $Y[j]$ is
a Young diagram, then $\B(Y[j])$ is the set of semistandard
tableaux of shape $Y[j]$. If $Y[j]$ is {\it not} a Young diagram,
we define $\B(Y[j]) = \emptyset$. More generally, $Y[j_1, \cdots,
j_{N}]$ is the diagram obtained from $Y[j_1, \cdots, j_{N-1}]$ by
adding a box at the $j_{N}$-th row and $\B(Y[j_1, \cdots, j_N])$
is the set of semistandard tableaux of shape $Y[j_1, \cdots, j_N]$
if $Y[j_1, \cdots, j_k]$ is a Young diagram for all $k=1, \cdots,
N$. We define $\B(Y[j_1, \cdots, j_N]) = \emptyset$ if $Y[j_1,
\cdots j_k]$ is not a Young diagram for some $k=1, \cdots, N$.

\vskip 3mm

\begin{example}
For $Y$= \raisebox{-0.5\height}{
\begin{texdraw}
\drawdim em \setunitscale 0.15  \linewd 0.4  \sq \sq \move(0 -10)
\sq \sq \end{texdraw}}\ , $Y[1,3,1,2]$ is obtained as follows:
\vskip3mm

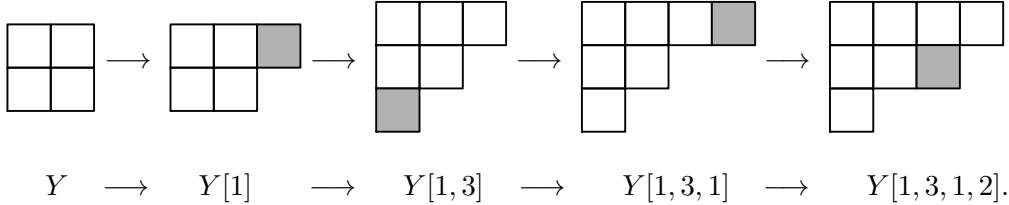
\raisebox{-0.5\height}{
\begin{texdraw}
\drawdim em \setunitscale 0.15  \linewd 0.4  \sq \sq \move(0 -10)
\sq  \sq
\end{texdraw}}
$\longrightarrow$ \raisebox{-0.5\height}{
\begin{texdraw}
\drawdim em \setunitscale 0.15  \linewd 0.4  \sq \sq \bsq \move(0
-10) \sq  \sq  \end{texdraw}} $\longrightarrow$
\raisebox{-0.5\height}{
\begin{texdraw}
\drawdim em \setunitscale 0.15  \linewd 0.4  \sq \sq \sq \move(0
-10) \sq \sq  \move(0 -20) \bsq \end{texdraw}} $\longrightarrow$
\raisebox{-0.5\height}{
\begin{texdraw}
\drawdim em \setunitscale 0.15  \linewd 0.4  \sq \sq \sq \bsq
\move(0 -10) \sq  \sq  \move(0 -20) \sq \end{texdraw}}
$\longrightarrow$ \raisebox{-0.5\height}{
\begin{texdraw}
\drawdim em \setunitscale 0.15  \linewd 0.4  \sq \sq \sq \sq
\move(0 -10) \sq \sq  \bsq  \move(0 -20) \sq \end{texdraw}}

\vskip5mm \hspace{5 mm} $Y$ \hspace{2 mm} $\longrightarrow$
\hspace{4 mm} $Y[1]$  \hspace{5 mm} $\longrightarrow$ \hspace{5
mm}$Y[1,3]$ \hspace{2 mm} $\longrightarrow$ \hspace{5 mm}
$Y[1,3,1]$ \hspace{3 mm}$\longrightarrow$ \hspace{5 mm}
$Y[1,3,1,2]$.

\end{example}

With this notation, Nakashima obtained the following decomposition
of the tensor product of $U_q(\gl(r))$-crystals.

\begin{proposition} \cite{N}
Let $Y$ and $W$ be Young diagrams with at most $r$ rows. Then
there exists a $U_q(\gl(r))$-crystal isomorphism
\begin{equation} \label{eq:gl(r)-decomp}
    \B(Y) \otimes \B(W) \cong \displaystyle\bigoplus_{\substack{T \in \B(W),\\
R_{FE}(T)=\boxed{j_1} \otimes \cdots \otimes \boxed{j_N}}}^{}
\B(Y[j_1,\cdots,j_N]).
\end{equation}
\end{proposition}

By the same argument in \cite[Proposition 4.13]{KK}, one can show
that $Y[j_1, \cdots, j_N]$ are the same for all admissible reading
$R_{A}(T)= \boxed{j_1} \otimes \cdots \otimes \boxed{j_N}$. Hence
the tensor product decomposition \eqref{eq:gl(r)-decomp} can be
restated as
\begin{equation} \label{eq:gl(r)-decomp2}
    \B(Y) \otimes \B(W) \cong \displaystyle\bigoplus_{\substack{T \in \B(W),\\
R_{A}(T)=\boxed{j_1} \otimes \cdots \otimes \boxed{j_N}}}^{}
\B(Y[j_1,\cdots,j_N]),
\end{equation}
where $R_{A}$ is an arbitrary admissible reading.

Motivated by this, we make the following definitions. Let $Y$,
$W$, $Z$ be the Young diagrams with at most $r$ rows such that
$|Y|+|W|=|Z|$, and let $A$ be an admissible order on $W$. We
define $\B(W)_{Y}^{Z}[A]$ to be the set of semistandard tableaux
of shape $W$ satisfying the following conditions:
\begin{equation} \label{eq:LRcrystal}
\begin{aligned}
& \text{if $R_A(T)=\boxed{j_1} \otimes \cdots \otimes
\boxed{j_N}$,
then} \\
& \ \ \text{(i) $Y[j_1, \cdots, j_k]$ is a Young diagram for all
$k=1,
\cdots, N$}, \\
& \ \ \text{(ii) $Y[j_1, \cdots, j_N]= Z$.}
\end{aligned}
\end{equation}

Since the decomposition \eqref{eq:gl(r)-decomp2} is unique, the
set $\B(W)_{Y}^{Z}[A]$ are the same for all admissible orders $A$
on $W$. Hence we may write
$$\B(W)_{Y}^{Z} = \B(W)_{Y}^{Z}[A]$$
for any admissible order $A$ on $W$.

\begin{definition} \hfill

(a) A semistandard tableau $T$ in $\B(W)_{Y}^{Z}$ is called a {\it
$U_q(\gl(r))$-Littlewood-Richardson tableau} associated with the
triple $(Y, W, Z)$.

(b) The number $c_{Y, W}^{Z} = | \B(W)_{Y}^{Z}|$ is called the
{\it $U_q(\gl(r))$-Littlewood-Richardson coefficient} associated
with $(Y, W, Z)$.

\end{definition}

\vskip 3mm

Now we explain the main result of \cite{NS}. Let $Y$ be a Young
diagram and let $T$ be a semistandard tableau in $\B(Y)$. We
denote by $T_{ij}$ the $(i,j)$-entry of $T$. For each $k \in \N$,
we define
$$T^{(k)} = \{(i,j) \in Y \mid T_{ij}=k \}.$$
Since no entry can occur more than once in any column of $T$, we
may write
$$T^{(k)} = \{(a_1, b_1), \cdots, (a_r, b_r) \},$$
where $a_1 \le \cdots \le a_r$, \ $b_1 > \cdots > b_r$. Define
$$p(T; a_i, b_i) = i \quad \text{for} \ (a_i, b_i) \in T^{(k)}.$$
Hence if $T$ is a semistandard tableau in $\B(Y)$ with $T_{ij}=k$,
then $p(T; i,j)$ is equal to the number of boxes with entry $k$
that lie in the right of $T_{ij}$ (including $T_{ij}$ itself).

\vskip 5mm

\begin{example} For $T=$  \raisebox{-0.5\height}{
\begin{texdraw}
\drawdim em \setunitscale 0.15  \linewd 0.4  \sqo \sqo \sqt \move
(0 -10) \sqt \sqt \vskip0.5em
\end{texdraw}} , we have
$$T^{(1)}=\{ (1,2), (1,1)\}, \ \ T^{(2)}= \{ (1,3), (2,2), (2,1) \}, $$
and $$p(T;1,1) =2, \ \ p(T;1,2)=1, \ \ p(T;1,3)=1, \ \ p(T;2,1)
=3, \ \ p(T;2,2)=2.$$
\end{example}

\vskip 3mm

\begin{proposition}\cite{NS} \label{prop:NS} Let $Y,W,Z$ be Young diagrams with at most $r$ rows such that
$|Y|+|W|=|Z|$, and let $A$, $A'$ be admissible orders on $Z/Y$,
$W$, respectively. Then there exists a natural bijection
$$\Phi: \P(W,Z/Y; A, A') \longrightarrow
\B(W)_{Y}^{Z}[A']$$ defined by
\begin{equation} \label{eq:phi}
\Phi(f)_{ij} = f_1(i,j) \ \ \text{for} \ f=(f_1, f_2) \in \P(W,
Z/Y;A,A').
\end{equation}
The inverse map
$$\Psi: \B(W)_{Y}^{Z}[A'] \longrightarrow \P(W, Z/Y; A, A')$$ is
given by
\begin{equation} \label{eq:psi} \Psi(T)(i,j) =(T_{ij},
Y_{T_{ij}} + p(T; i,j)) \ \ \text{for} \ T \in \B(W)_{Y}^{Z}[A'].
\end{equation}
\end{proposition}

Let $Y$, $Z$ be Young diagrams with at most $r$ rows and let $W$
be a skew Young diagram such that $|Y|+|W|=|Z|$. We define
$\B(W)_{Y}^{Z}[A]$ to be the set of all semistandard skew tableaux
$T$ in $\B(W)$ satisfying the condition \eqref{eq:LRcrystal}. For
a semistandard skew tableau $T$ of shape $W$, we define $p(T;
i,j)$ to be the number of boxes with entry $T_{ij}$ lying in the
right of the $(i,j)$-position (including the box at the
$(i,j)$-position). Our first main result is the following
generalization of Proposition \ref{prop:NS}.

\begin{theorem} \label{thm:skew_case}

Let $Y,Z$ be Young diagrams with at most $r$ rows and let $W$ be a
skew Young diagram such that $|Y|+ |W|=|Z|$. Let $A$ and $A'$ be
arbitrary admissible orders on $Z/Y$ and $W$, respectively. Then
there exist natural bijections
$$\begin{aligned}
& \Phi: \P(W, Z/Y; A, A') \longrightarrow \B(W)_{Y}^{Z}[A'], \\
& \Psi: \B(W)_{Y}^{Z}[A'] \longrightarrow \P(W, Z/Y; A, A')
\end{aligned}$$
defined by \eqref{eq:phi} and \eqref{eq:psi}, which are the
inverses to each other.
\end{theorem}

\begin{proof} Our proof follows the outline given in \cite{NS}. The
key ingredient of our generalization is the following almost
self-obvious lemma on skew Young diagrams.

\begin{lemma} \label{lem:skew}
Let $W$ be a skew Young diagram. If $(a,b), (c,d) \in W$ and $a
\le c, b \le d$, then every $(x,y)$ satisfying
$$(a, b) \le_{P} (x,y) \le_{P} (c,d)$$
lies in W.
\end{lemma}

We now proceed to prove our theorem in 3 steps.

\vskip 3mm

\noindent {\bf Step 1:} {\it The map $\Phi$ is well-defined.}

Let $f=(f_1, f_2) \in \P(W, Z/Y; A, A')$. We first show that
$\Phi(f)$ is a semistandard skew tableau of shape $W$. That is, we
show

\ \ (i) $f_1(i,j) < f_1(i+1, j)$ for all $(i,j) \in W$,

\ \ (ii) $f_1(i,j) \le f_1(i, j+1)$ for all $(i,j) \in W$.

The condition (i) can be verified by the same argument in
\cite[Proposition 5.1]{NS}. We will prove the condition (ii) using
the induction on $i$. Let
$$i_{0} = \min \{ i \in \N \mid (i,j) \in W, \ (i, j+1) \in W \}.$$
Thus $W$ has more than two boxes in the $i_0$-th row for the first
time.

If $i=i_0$, suppose on the contrary that $f_1(i_0, j) > f_1(i_0,
j+1)$. Then by \cite[Lemma 5.2]{NS}, we have $f_2(i_0, j) >
f_2(i_0, j+1)$. Moreover, by \cite[Lemma 5.3]{NS}, there exists a
unique $(k,l) \in W$ satisfying
\begin{equation} \label{eq:NS5.3}
k<i_0, \ l \le j, \ \ f_1(k,l) = f_1(i_0, j+1), \ f_2(k,l) =
f_2(i_0, j).
\end{equation}
Since
$$(k,l) \le_{P} (i_0 -1, j) \le_{P} (i_0, j) \ \ \text{for} \ (k,l),
(i_0, j) \in W,$$ by Lemma \ref{lem:skew}, we have $(i_0 -1, j)
\in W$. Applying Lemma \ref{lem:skew} to $(i_0-1, j)$ and $(i_0,
j+1)$, we get $(i_0-1, j+1) \in W$. Hence we have $(i_0 -1, j) \in
W$ and $(i_0 -1, j+1) \in W$, which is a contradiction to the
minimality of $i_0$. Therefore, we conclude $f_1(i_0, j) \le
f_1(i_0, j+1)$ for all $j$ with $(i_0, j), (i_0, j+1) \in W$.

If $i>i_0$, by a similar argument in the proof of
\cite[Proposition 5.1 (i)]{NS} using Lemma \ref{lem:skew} whenever
necessary, one can verify the condition (ii). Hence $\Phi(f)$ is a
semistandard skew tableau of shape $W$.

Using almost the same argument in the proof of \cite[Proposition
5.1 (ii)]{NS}, it is straightforward to verify that $\Phi(f)$
satisfies the condition \eqref{eq:LRcrystal} with respect to $A'$.
Therefore, $\Phi$ is well-defined.

\vskip 3mm

\noindent {\bf Step 2:} {\it The map $\Psi$ is well-defined.}

Let $T \in \B(W)_{Y}^{Z}[A']$. One can easily verify that
$\Psi(T)$ is a bijection from $W$ to $Z/Y$. It remains to show
$\Psi(T)$ is an $(A, A')$-admissible picture. For this purpose,
one can verify that almost all the arguments in the proof of
\cite[Proposition 6.1]{NS} work in our case as well. The only
difference is that, for a skew Young diagram $W$, the
$U_q(\gl(r))$-crystal $\B(W)$ may not be connected. However, since
the boxes in different connected components are not comparable
with respect to $P$, it suffices to show that $\Psi(T)$ is
$PA$-standard on each connected component, which can be checked in
a straightforward manner combining Lemma \ref{lem:skew} and the
proof of \cite[Proposition 6.1 (3)]{NS}.

\vskip 3mm

\noindent {\bf Step 3:} {\it $\Phi$ and $\Psi$ are inverses to
each other.}

The same argument in \cite[Section 7]{NS} works in our case as
well.
\end{proof}

\vskip 1cm


\section{$U_q(\gl(m,n))$-Littlewood-Richardson tableaux}

In this section, we will prove the main result of this paper. We
will define the notion of {\it
$U_q(\gl(m,n))$-Littlewood-Richardson tableaux} and show that
there exists a natural bijection between the set of admissible
pictures and the set of $U_q(\gl(m,n))$-Littlewood-Richardson
tableaux. One may refer to \cite{BKK} and \cite{KK} for more
details on $U_q(\gl(m,n))$-crystals.

Let $$\mathfrak{B}: \ \ \boxed{1} \overset{1} \longrightarrow
\boxed{2} \overset{2} \longrightarrow \cdots \overset{m-1}
\longrightarrow \boxed{m} \overset{m} \longrightarrow
\boxed{\bar{1}} \overset{\bar {1}} \longrightarrow \boxed{\bar{2}}
\overset{\bar{2}} \longrightarrow \cdots \overset{\overline{n-1}}
\longrightarrow \boxed{\bar{n}}$$ be the crystal of the vector
representation of $U_q(\gl(m,n))$. We define an ordering on
$\mathfrak{B}$ by
$$1 < 2 < \cdots < m < \bar{1} < \bar{2} < \cdots < \bar{n},$$
and set $\mathfrak{B}_{+}=\{1,2, \cdots, m\}$,
$\mathfrak{B}_{-}=\{\bar{1},\bar{2}, \cdots, \bar{n}\}$.

\begin{definition} Let $Y$ be a skew Young diagram. A {\it
$\gl(m,n)$-semistandard skew tableau of shape $Y$} is a tableau
$T$ obtained from $Y$ by filling the boxes with entries from
$\mathfrak{B}$ satisfying the following conditions:

(i) the entries in each row and column are weakly increasing,

(ii) the entries in $\mathfrak{B}_{+}$ are strictly increasing in
each column,

(iii) the entries in $\mathfrak{B}_{-}$ are strictly increasing in
each row.
\end{definition}

We denote by $\mathfrak{B}(Y)$ the set of all
$\gl(m,n)$-semistandard tableaux of shape $Y$. In \cite{BKK}, it
was shown that an admissible reading provides $\mathfrak{B}(Y)$
with a $U_q(\gl(m,n))$-crystal structure and it does not depend on
the choice of admissible reading.

A Young diagram $Y$ is called an {\it $(m,n)$-hook Young diagram}
if the number of boxes in $(m+1)$-th row is at most $n$; i.e.,
there is no box in the $(m+1, n+1)$-th position. We denote by
$H(m,n)$ the set of all $(m,n)$-hook Young diagrams. Note that a
Young diagram $Y$ can be made into a $\gl(m,n)$-semistandard
tableau if and only if $Y$ is an $(m,n)$-hook Young diagram. To
describe the decomposition of the tensor product of
$U_q(\gl(m,n))$-crystals, we need the following definitions.

Let $Y$ be a skew Young diagram with an admissible order $A$ and
let $Q$ be a semistandard skew tableau of shape $Y$ with entries
in $\N$. The {\it word of $Q$ with respect to $A$} is defined to
be
$$w_{A}(Q) = (i_1, i_2, \cdots, i_N), $$
where $R_{A}(Q) = \boxed{i_1} \ot \cdots \ot \boxed{i_N}$, and the
{\it content} of $Q$ is defined to be
$$\text{cont}(Q)= (\mu_i)_{i \in \N},$$
where $\mu_i$ is the number of $i$'s in $Q$. A finite sequence
$(i_1, \cdots, i_N)$ is called a {\it lattice permutation} if for
every $i \in \N$ and $k$ with $1 \le k \le N$, the number of
occurrences of $i$ in $(i_1, \cdots, i_k)$ is greater than or
equal to the number of occurrences of $i+1$ in $(i_1, \cdots,
i_k)$.

\begin{definition}
Let $Y$, $W$, $Z$ be $(m,n)$-hook Young diagrams with $|Y|+|W|
=|Z|$ and let $A$ be an admissible order on $Z/Y$. A {\it
$U_q(\gl(m,n))$-Littlewood-Richardson tableau associated with
$(Y,W,Z)$ and $A$} is a semistandard skew tableau $Q$ with entries
in $\N$ satisfying the following conditions:

(i) $\sh(Q) = Z/Y$,

(ii) $\text{cont}(Q)=W$,

(iii) $w_{A}(Q)$ is a lattice permutation.

\end{definition}

We denote by $LR(Y,W)^{Z}[A]$ the set of all
$U_q(\gl(m,n))$-Littlewood-Richardson tableaux associated with
$(Y, W, Z)$ and $A$. Using the insertion scheme for
$U_q(\gl(m,n))$-crystals, it was shown in \cite[Proposition
4.13]{KK} that $LR(Y, W)^{Z}[A]$ are the same for all admissible
orders on $Z/Y$. The number $N_{Y,W}^{Z} = |LR(Y,W)^{Z}[A]|$ is
called the {\it $U_q(\gl(m,n))$-Littlewood-Richardson coefficient}
associated with $(Y,W,Z)$. Moreover, \cite[Theorem 4.16]{KK}
yields the decomposition of the tensor product of
$U_q(\gl(m,n))$-crystals:
\begin{equation} \label{eq:gl(m,n)-decomp}
\mathfrak{B}(Y) \ot \mathfrak{B}(W) \cong \bigoplus_{Z \in H(m,n)}
\mathfrak{B}(Z)^{\oplus N_{Y,W}^{Z}}.
\end{equation}

We now state and prove our main theorem.

\begin{theorem} \label{thm:main}
Let $Y$, $W$, $Z$ be $(m,n)$-hook Young diagrams with
$|Y|+|W|=|Z|$, and let $A$, $A'$ be admissible orders on $W$,
$Z/Y$, respectively. Then there exists a natural bijection
$$\widetilde{\Phi}: \P(Z/Y, W; A, A') \longrightarrow LR(Y,W)^{Z}[A']$$
defined by
\begin{equation} \label{eq:tphi}
\widetilde{\Phi}(f)_{ij} = f_1(i,j) \ \ \text{for} \ f=(f_1, f_2)
\in \P(Z/Y, W; A, A').
\end{equation}
The inverse map $$\widetilde{\Psi}: LR(Y,W)^{Z}[A']
\longrightarrow \P(Z/Y,W; A, A')$$ is given by
\begin{equation} \label{eq:tpsi}
\widetilde{\Psi}(Q)(i,j) = (Q_{ij}, p(Q;i,j)) \ \ \text{for} \ Q
\in LR(Y,W)^{Z}[A'].
\end{equation}
\end{theorem}

\begin{proof}
Write $W=(W_1 \ge W_2 \ge \cdots \ge W_k >0)$ and $Z=(Z_1 \ge Z_2
\ge \cdots \ge Z_h >0)$. Set $r=\max(k,h)$ and consider the set
$\B(Z/Y)$ of semistandard skew tableaux of shape $Z/Y$ with
entries in $\{1, 2, \cdots, r \}$. We claim
$$LR(Y,W)^{Z}[A'] = \B(Z/Y)_{\emptyset}^{W}[A'].$$

Let $Q \in LR(Y,W)^{Z}[A']$. Then $Q$ is a semistandard skew
tableau of shape $Z/Y$ with entries in $\N$ such that
$\text{cont}(Q)=W$ and $w_{A'}(Q)$ is a lattice permutation.
Write $R_{A'}(Q)= \boxed{i_1} \ot \cdots \ot \boxed{i_N}$. Since
$w_{A'}(Q)=(i_1, \cdots ,i_N)$ is a lattice permutation,
$\emptyset[i_1, \cdots, i_r]$ is a Young diagram for all $r=1,
\cdots, N$. Moreover, since $\text{cont}(Q)=W$, we have
$\emptyset[i_1, \cdots, i_N] =W$. Hence $Q \in
\B(Z/Y)_{\emptyset}^{W}[A']$.

Conversely, if $T \in \B(Z/Y)_{\emptyset}^{W}[A']$, then a similar
argument shows $T \in LR(Y,W)^{Z}[A']$, which proves our claim.

By Theorem \ref{thm:skew_case}, there exist natural bijections
\begin{equation*}
\begin{aligned}
& \Phi: \P(Z/Y, W; A, A') \longrightarrow
\B(Z/Y)_{\emptyset}^{W}[A'], \\
& \Psi: \B(Z/Y)_{\emptyset}^{W}[A'] \longrightarrow \P(Z/Y, W;
A,A')
\end{aligned}
\end{equation*}
defined by
\begin{equation*}
\begin{aligned}
& \Phi(f)_{ij}=f_1(i,j) \ \ \text{for} \ f = (f_1, f_2) \in
\P(Z/Y,
W; A, A'), \\
& \Psi(Q)(i,j) = (Q_{ij}, p(Q; i,j)) \ \ \text{for} \ Q \in
\B(Z/Y)_{\emptyset}^{W}[A'].
\end{aligned}
\end{equation*}
Since $LR(Y,W)^{Z}[A'] = \B(Z/Y)_{\emptyset}^{W}[A']$, we are
done.
\end{proof}

\vskip 3mm

\begin{corollary} \label{crystal}
Let $Y, W, Z$ be $(m,n)$-hook Young diagrams and let $A$, $A'$ be
admissible orders on $W$, $Z/Y$, respectively. Write $W=(W_1 \ge
W_2 \ge \cdots \ge W_k >0)$, $Z=(Z_1 \ge Z_2 \ge \cdots \ge Z_h
>0)$ and set $r=\max(k,h)$. Then there exists a natural bijection
$$\widehat{\Phi}: LR(Y,W)^{Z}[A'] \longrightarrow
\B(W)_{Y}^{Z}[A]$$ between the set of
$U_q(\gl(m,n))$-Littlewood-Richardson tableaux and the set of
$U_q(\gl(r))$-Littlewood-Richardson tableaux. In particular, we
have $N_{Y,W}^{Z}=c_{Y,W}^{Z}$.
\end{corollary}

\begin{proof}
By the definition of admissible pictures, the map
\beq \begin{array}{ccccc} & \Omega: &\P(Z/Y,W;A,A') & \ra & \P(W,Z/Y;A',A)\\
              & &    f  & \map & f^{-1}
\end{array} \eeq
is a bijection. Hence the composition $\widehat{\Phi} = \Phi \circ
\Omega \circ \widetilde{\Psi}$ is the desired bijection.
\end{proof}

\begin{example}
Let $Y=(5,2,1), \ W=(3,2,2,1)$ and $Z=(6,4,2,2,2)$ be $(3,3)$-hook
Young diagrams. We present the 1-1 correspondence between
$LR(Y,W)^Z[ME]$ for $\gl(3,3)$ and $\B(W)_{Y}^{Z}[ME]$ for
$\gl(5)$ given in the proof of Corollary \ref{crystal}.
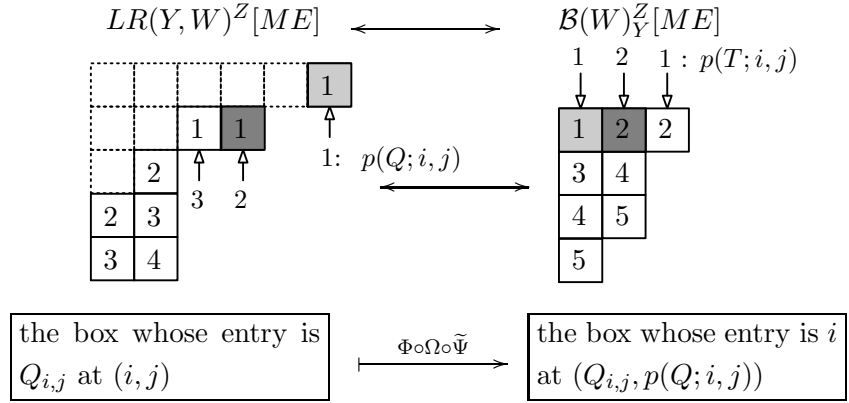
\begin{center}

  \hskip2em $LR(Y,W)^Z[ME]$         \xymatrix{\ar@{<->}[rr] &&}        $\ \ \ \ \mathcal{B}(W)_{Y}^{Z}[ME]$ \vskip3mm

 \raisebox{-0.5\height}{
\begin{texdraw}
\drawdim em \setunitscale 0.15  \linewd 0.4  \dsq \dsq \dsq  \dsq
\dsq  \sqo \move(0 -10) \dsq \dsq \sqo \sqo \move(0 -20) \dsq \sqt
\move(0 -30) \sqt \sqth \move(0 -40) \sqth \sqf
 \move(55 -7) \arrowheadsize l:3 w:2 \avec(55 0) \htext(53 -15){\small{$1$}: \  $p(Q;i,j)$}
  \move(35 -17) \arrowheadsize l:3 w:2 \avec(35 -10) \htext(33 -23){\small{2}}
  \move(25 -17) \arrowheadsize l:3 w:2 \avec(25 -10) \htext(23 -23){{\small 3}}

  \move(50 0) \bsegment \move(0
0)\lvec(10 0)\lvec(10 10)\lvec(0 10)\lvec(0 0) \lfill f:0.8
\esegment

  \move(30 -10) \bsegment \move(0
0)\lvec(10 0)\lvec(10 10)\lvec(0 10)\lvec(0 0) \lfill f:0.5
\esegment

 \vskip1em
\end{texdraw}}   \xymatrix{\ar@{<->}[rr] &&} \raisebox{-0.5\height}{
\begin{texdraw}
\drawdim em \setunitscale 0.15  \linewd 0.4  \sqo \sqt \sqt
\move(0 -10) \sqth \sqf \move(0 -20) \sqf \sqfi \move(0 -30) \sqfi

 \move(0 0) \bsegment \move(0
0)\lvec(10 0)\lvec(10 10)\lvec(0 10)\lvec(0 0) \lfill f:0.8
\esegment

  \move(10 0) \bsegment \move(0
0)\lvec(10 0)\lvec(10 10)\lvec(0 10)\lvec(0 0) \lfill f:0.5
\esegment

\move(5 18) \arrowheadsize l:3 w:2 \avec(5 10) \htext(3
20){{\small 1}}
 \move(15 18) \arrowheadsize l:3 w:2 \avec(15 10)  \htext(13 20){{\small 2}}
\move(25 18) \arrowheadsize l:3 w:2 \avec(25 10)     \htext(23
18){{\small 1 : $p(T;i,j)$}}

\vskip1em
\end{texdraw}}

\hskip3em \fbox{\parbox[c]{4cm}{the box whose entry is $Q_{i,j}$
at $(i,j)$}}  \xymatrix{\ar@{|->}^-{\Phi \circ \Omega \circ
\widetilde{\Psi}}[rr] &&}  \fbox{\parbox[c]{4cm}{the box whose
entry is $i$ at $(Q_{i,j}, p(Q;i,j))$}} \end{center}

\vskip 2mm

For example,

\begin{center}
\hskip2em {\parbox[c]{4cm}{the box whose entry is 1 and
$p(Q;1,6)=1$ at (1,6)}}  \xymatrix{\ar@{|->}[rr] &&}
{\parbox[c]{4cm}{ the box whose entry is 1 at (1,1)}} \vskip3mm
\hskip2em {\parbox[c]{4cm}{the box whose entry is 1 and
$p(Q;2,4)=2$ at (2,4)}}  \xymatrix{\ar@{|->}[rr] &&}
{\parbox[c]{4cm}{ the box whose entry is 2 at (1,2)}}

\vskip1em

\hskip3em \fbox{\parbox[c]{4cm}{the box whose entry is $i$ at
$(T_{i,j},Y_{T_{i,j}}+p(T;i,j))$}}
\xymatrix{\ar@{<-|}^-{\widetilde{\Phi} \circ \Omega \circ
\Psi}[rr] &&} \fbox{\parbox[c]{4cm}{the box whose entry is
$T_{i,j}$ at $(i,j)$}}
\end{center}

\vskip 2mm

For instance,

\begin{center}
\hskip2em {\parbox[c]{4cm}{the box whose entry is 1 at (1,5+1)}}
\xymatrix{\ar@{<-|}[rr] &&}   {\parbox[c]{4cm}{ the box whose
entry is 1 and $p(T;1,1)=1$ at (1,1)}} \vskip3mm \hskip2em
{\parbox[c]{4cm}{the box whose entry is 1 at (2,2+2)}}
\xymatrix{\ar@{<-|}[rr] &&}   {\parbox[c]{4cm}{ the box whose
entry is 2 $p(T;1,2)=2$ at (1,2)}}
\end{center}

\vskip3mm

Similarly, we have the following correspondence: \vskip3mm

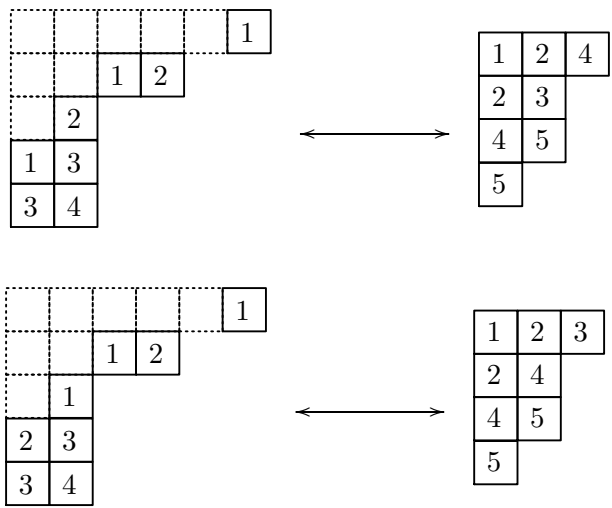
\begin{center}
\raisebox{-0.5\height}{
\begin{texdraw}
\drawdim em \setunitscale 0.15  \linewd 0.4  \dsq \dsq \dsq  \dsq
\dsq  \sqo \move(0 -10) \dsq \dsq \sqo \sqt \move(0 -20) \dsq \sqt
\move(0 -30) \sqo \sqth \move(0 -40) \sqth \sqf \vskip1em
\end{texdraw}}    \xymatrix{\ar@{<->}[rr] &&}
\raisebox{-0.5\height}{
\begin{texdraw}
\drawdim em \setunitscale 0.15  \linewd 0.4  \sqo \sqt \sqf  \move
(0 -10) \sqt \sqth \move(0 -20) \sqf \sqfi \move(0 -30) \sqfi
\vskip1em
\end{texdraw}}

\vskip1em \raisebox{-0.5\height}{\begin{texdraw} \drawdim em
\setunitscale 0.15  \linewd 0.4  \dsq \dsq \dsq  \dsq   \dsq  \sqo
\move(0 -10) \dsq \dsq \sqo \sqt \move(0 -20) \dsq \sqo  \move(0
-30) \sqt \sqth \move(0 -40) \sqth \sqf  \vskip1em \end{texdraw}}
\xymatrix{\ar@{<->}[rr] &&}  \raisebox{-0.5\height}{
\begin{texdraw}
\drawdim em \setunitscale 0.15  \linewd 0.4  \sqo \sqt \sqth \move
(0 -10) \sqt \sqf \move(0 -20) \sqf \sqfi \move(0 -30) \sqfi
\vskip1em
\end{texdraw}}

\end{center}

\end{example}

\begin{remark}
The correspondence in Corollary \ref{crystal} is the same as that
of Theorem C in the appendix of \cite{N}.
\end{remark}

\vskip 3mm

\begin{acknowledgement}
The authors would like to thank Professor Jae-Hoon Kwon for many
valuable advices and stimulating discussions.
\end{acknowledgement}

\vskip 5mm


\end{document}